\documentclass[reqno,10pt]{amsart}
\usepackage[active]{srcltx}
\usepackage{color,xcolor}
\usepackage{amsmath,amsthm,amssymb}
\usepackage{epsfig}
\usepackage{graphicx}
\usepackage{amssymb}
\usepackage{latexsym}
\usepackage{pdfpages}
\usepackage{enumerate}
\usepackage{url}
\usepackage{appendix}
\usepackage{pgf}

\usepackage{enumitem}

\usepackage{ulem}

\usepackage{ifpdf}

\ifpdf 
  \usepackage[hidelinks]{hyperref}
\else 
\fi 

 \usepackage[usenames,dvipsnames]{pstricks}
 \usepackage{epsfig}
 \usepackage{pst-grad} 
 \usepackage{pst-plot} 

\usepackage{color}

\newtheorem{theorem}{Theorem}[section]
\newtheorem{lemma}[theorem]{Lemma}

\newtheorem{proposition}[theorem]{Proposition}

\newtheorem*{theorem*}{\it Theorem}

\numberwithin{equation}{section}

\def\1{\raisebox{2pt}{\rm{$\chi$}}}

\def\R{\mathbb{R}}

\def\XXint#1#2#3{{\setbox0=\hbox{$#1{#2#3}{\int}$}
     \vcenter{\hbox{$#2#3$}}\kern-.5\wd0}}

\definecolor{violet(ryb)}{rgb}{0.53, 0.0, 0.69}
\usepackage{mathtools}

\usepackage{geometry}
\begin{document}

\title[$\ell_1$ double-bubbles]{\bf The $\ell_1$ double-bubble problem in three dimensions}

\author[M. Friedrich]{Manuel Friedrich} 
\address[Manuel Friedrich]{Department of Mathematics, Friedrich-Alexander Universit\"at Erlangen-N\"urnberg. Cauerstr.~11,
    D-91058 Erlangen, Germany, \& Mathematics M\"{u}nster,  
University of M\"{u}nster, Einsteinstr.~62, D-48149 M\"{u}nster, Germany}
\email{manuel.friedrich@fau.de}

\author[W. G\'orny]{Wojciech G\'orny}
\address[Wojciech G\'orny]{Faculty of Mathematics, Informatics and
  Mechanics, University of Warsaw, Banacha 2, 02-097 Warsaw, Poland
  and Faculty of Mathematics, University of
  Vienna, Oskar-Morgenstern-Platz 1, A-1090 Vienna, Austria}
\email{wojciech.gorny@univie.ac.at}
\urladdr{\url{https://www.mat.univie.ac.at/~wgorny}}

\author[U. Stefanelli]{Ulisse Stefanelli}
\address[Ulisse Stefanelli]{Faculty of Mathematics, University of
  Vienna, Oskar-Morgenstern-Platz 1, A-1090 Vienna, Austria,
Vienna Research Platform on Accelerating
  Photoreaction Discovery, University of Vienna, W\"ahringerstra\ss e 17, 1090 Vienna, Austria,
 \& Istituto di
  Matematica Applicata e Tecnologie Informatiche {\it E. Magenes}, via
  Ferrata 1, I-27100 Pavia, Italy
}
\email{ulisse.stefanelli@univie.ac.at}
\urladdr{\url{http://www.mat.univie.ac.at/~stefanelli}}


\setcounter{tocdepth}{1}


\begin{abstract}
We characterize the unique minimizer of the three-dimensional
double-bubble problem with respect to the $\ell_1$-norm for volume ratios between $1/2$ and $2$.
\end{abstract}

\maketitle
\thispagestyle{empty}


\section{Introduction}

The double-bubble problem consists in determining the optimal pair of sets of given volumes minimizing the total surface. In the classical Euclidean setting, optimal configurations are pairs of
regions enclosed by  three spherical caps, meeting at a $2\pi/3$ angle. This was first proved in the planar case in \cite{Foisy}, then extended in \cite{Hutchings} to three dimensions, and finally to all dimensions in \cite{Reichardt}. Besides the Euclidean case, double-bubble problems have been considered in a variety of different settings, including hyperbolic spaces \cite{Corneli2,Corneli,Cotton,Masters}, hyperbolic surfaces \cite{Boyer} and cones \cite{Lopez,Morgan}, the three-dimensional torus \cite{Carrion,Corneli0}, the Gau\ss\ space \cite{Corneli2,Milman}, and the anisotropic Grushin plane \cite{Franceschi}.

This note is concerned with the three-dimensional double-bubble
problem for the $\ell_1$-norm. Given $v \in \mathbb{R}^3$, we denote by $|v|_1$ its $\ell_1$-norm 
$$ |v|_1 = |v_1| + |v_2| + |v_3|$$
(we will later use the same notation for the $\ell_1$-norm of vectors in $\mathbb{R}^2$). For any $\mathcal{H}^2$-rectifiable subset $F \subset \R^3$, we denote the corresponding $\ell_1$-surface by 
$${\ell_1}(F) = \int_{F}   |\nu_F|_1 \, {\rm d}\mathcal{H}^2,  $$
where $\nu_F$ denotes the  measure-theoretical normal to $F$ and $\mathcal{H}^n$ stands for the
$n$-dimensional Hausdorff measure. We  consider sets of finite perimeter
$G  \subset \R^3$ \cite{Ambrosio-Fusco-Pallara} and use the
fact that their so-called {\it reduced boundary} $\partial^* G
$  is a $\mathcal{H}^2$-rectifiable set. To each configuration $(A,B)$ consisting of two three-dimensional sets of finite perimeter, we associate the energy
\begin{equation*}
E(A,B) := \ell_1(\partial^* A  )+\ell_1(\partial^* B  ) - \ell_1(\partial^* A \cap \partial^*B).
\end{equation*}
 This corresponds to the $\ell_1$-surface of the set $A\cup B$  together with the $\ell_1$-interface between the sets $A$ and $B$. In the following, we let  the volumes $V_A:=\mathcal{L}^3(A)$ and
$V_B:=\mathcal{L}^3(A)$ be fixed,  where
$\mathcal{L}^n$ denotes the $n$-dimensional Lebesgue measure.  Our main assumption is that the ratio 
  $ {V_B}/{V_A} $ belongs to $[1/2,2]$.

The {\it double-bubble problem} hence corresponds to
\begin{align}
\displaystyle \min \bigg\{ E(A,B) \, \colon \, \ A,\, B \subset \R^3 \ \text{of finite perimeter,} \ A \cap B= \emptyset, \ \mathcal{L}^3(A)=V_A, \ \mathcal{L}^3(B)=V_B \bigg\}.\label{eq:db_problem}
\end{align}
Our main result reads as follows. 

\begin{theorem}[Characterization of the minimizer]\label{thm:mainresult}
Letting $ {V_B}/{V_A} \in[1/2,2]$, the unique minimizer of the
double-bubble problem \eqref{eq:db_problem} are two cuboids sharing a
square face. 
Up to translation and
axis-preserving isometries, the minimizer can be specified as
\begin{align*}
  A &=   \Big[- \frac{V_A}{(2 (V_A+V_B) /3)^{2/3}},0 \Big] \times
      \Big[0,(2 (V_A+V_B) /3)^{1/3}\Big]^2  
      ,\\
  B&=  \Big[0, \frac{V_B}{(2 (V_A+V_B) /3)^{2/3}} \Big] \times  \Big[0, (2 (V_A+V_B) /3)^{1/3}\Big]^2 .
     \end{align*}
The minimal energy is given by
$$E(A,B) =  \bigg(3 \bigg(\frac{2}{3} \bigg)^{2/3} + 4 \bigg(\frac{3}{2} \bigg)^{1/3} \bigg)  (V_A+V_B)^{2/3}.$$
\end{theorem}
\begin{figure}[h]
  \centering
  \pgfdeclareimage[width=85mm]{fig2}{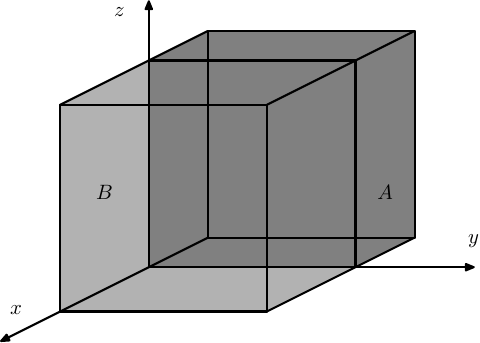}
  \pgfuseimage{fig2}
  \label{fig}
  \caption{The unique minimizer of the double-bubble
     problem \eqref{eq:db_problem}.}
\end{figure}

The minimality of the configuration in Theorem \ref{thm:mainresult} for this specific volume-ratio range has already been conjectured by {\sc Wecht, Barber, \& Tice} 
\cite{Wecht}.  In fact, by reducing the problem to cuboids, the optimality of $(A,B)$ from Theorem \ref{thm:mainresult} easily follows. Our aim here is to provide a proof of this conjecture,
starting from the most general setting of disjoint sets of finite perimeter.
Note that, for volume ratios $r$  smaller than $1/2$ or bigger  than  $2$, the
configuration in Figure \ref{fig} can be easily proved to be not optimal 
and the occurrence of different optimal configurations is conjectured \cite{Wecht}.

In the {\it planar case}, the characterization of optimal double-bubble configurations with respect to the $\ell_1$-norm in $\R^2$ is already well-known. The emergence of  three  different minimizers, depending on the volume ratio, has been discussed  by {\sc Morgan, French, \& Greenleaf} \cite{Morgan98}. A new proof of these results, based on different tools,
has been recently presented by {\sc Duncan, O'Dwyer, \& Procaccia} \cite{Duncan0}. The reach of the theory has been extended to the general setting of finite perimeter sets and to arbitrary interaction intensity in \cite{Double-Bubble-2D}. The continuous problem in $\R^2$  is  naturally connected with its discrete analogue on the ${\mathbb Z}^2$-lattice, which has also been studied \cite{Duncan,Double-Bubble}.  We further refer to \cite{Duncan3} for an analogous problem in the hexagonal norm. 

To our knowledge, our result is the first rigorous one 
for the $\ell_1$
double-bubble problem in three dimensions. 
 In fact,   our arguments  build  on the available 
understanding of the planar case by means of a {\it
  slicing} argument. We slice the minimizing configuration  with respect to a specific axis direction and we bound the 3D
energy $E$ in terms of an 
integral of the planar energies of the slices. Moving from the
knowledge of the exact value of the
2D minimal energy, see Proposition \ref{prop:chip}, this slicing approach
allows us to obtain an estimate of the minimal 3D energy $E$,
see Proposition \ref{prop:estimateforenergy}. This eventually turns
out to completely characterize optimal configurations.

The remainder of the paper is devoted to proving Theorem
\ref{thm:mainresult}. In particular, the proof of Theorem
\ref{thm:mainresult} is given in Section \ref{sec:mainproof}, based on
a few technical lemmas. These lemmas are then proved in
Section \ref{sec:proofsoflemmas}.

\section{Proof of the main result}\label{sec:mainproof}

As mentioned in the Introduction, the core step of the proof of
Theorem \ref{thm:mainresult} is that of
estimating from below the minimal value of the energy $E$ by taking
advantage of the characterization of minimizers in the planar case. We
hence start by recalling the 2D result in Subsection
\ref{sec:planar}. We then collect some notation and present a
crucial optimal bound  in Subsection
\ref{sec:notation}. After stating some technical lemmas in
Subsection \ref{sec:auxiliary}, the actual proof of Theorem
\ref{thm:mainresult} is given in Subsection~\ref{sec:proof:mainresult}. Eventually, the technical lemmas from
Subsection \ref{sec:auxiliary} are proved in Section \ref{sec:proofsoflemmas}.

\subsection{The planar case}\label{sec:planar}
Let us start by recalling the 2D result. Given a planar finite
perimeter set  $F_{2D}\subset {\mathbb R}^2$, we denote by $\partial^*_{2D}
F_{2D}\subset \R^2$ its {\it planar reduced boundary}, and by
$\nu=(\nu_1,\nu_2)$ the corresponding (measure-theoretic) {\it planar
  outer unit normal}. For all $\mathcal H^1$-rectifiable subsets $\varphi \subset \partial^*_{2D} F_{2D}$,  we
denote by
$$\ell_{1,2D}(\varphi) = \int_{\varphi}  (|\nu_1|  + |\nu_2|) \, {\rm d}\mathcal{H}^1$$
its length with respect to the $\ell_1$-norm in the plane.

We indicate the {\it minimal energy} of
a planar double bubble with regions of fixed areas $a,\, b  >0$ as
\begin{align}
E_{2D}(a,b) &:=    \min\Big\{\ell_{1,2D}(\partial^*_{2D}
  A_{2D}  )+\ell_{1,2D}(\partial^*_{2D} B_{2D} ) -
  \ell_{1,2D}(\partial^*_{2D} A_{2D} \cap \partial^*_{2D} B_{2D})\ \colon \nonumber \\
 & \qquad \qquad \qquad \text{$A_{2D},\, B_{2D}\subset \R^2$  of finite
   perimeter with} \nonumber\\
  &\qquad \qquad \qquad  A_{2D} \cap B_{2D}= \emptyset,  \
   \mathcal{L}^2(A_{2D})=a, \ \mathcal{L}^2(B_{2D})=b   \Big\}. \label{eq:2D_energy}
\end{align}
Define now the value
\begin{equation}r_*=\left(\frac{4(\sqrt{2}-1)}{1+2\sqrt{2}}\right)^2 \sim
0.1872957155.\label{rstar}
\end{equation}
 The main result in the planar case is the following \cite{Duncan,Double-Bubble-2D}.  

\begin{proposition}[Characterization of the planar minimizer]\label{thm:mainresult_planar}
Up to translations and axis-preserving isometries, the configurations $(A_{2D},B_{2D})$ realizing the minimum in \eqref{eq:2D_energy} are given by  
\begin{itemize}
\item Case $a/b\in[1/2,1]$
\begin{equation*}
A_{2D}=[-a/c,0]\times [0,c], \quad B_{2D}=[0,b/c] \times[0,c] \quad \text{with} \ \ c= \sqrt{\frac{2(a+b)}{3}},
\end{equation*}
    and corresponding energy $E_{2D}(a,b) = 2\sqrt{6}\sqrt{a+b}$;
 \item  Case $a/b\in[r_*,1/2]$ 
    \begin{equation*}
      A_{2D}= [-a/c,0]\times[0,c] +(0,\lambda) , \quad B_{2D}= [0,\sqrt{b}]^2 \quad
      \text{with} \ \ c= \sqrt{2a}
    \end{equation*}
     for some $\lambda \in[0, \sqrt{b}-c]$, 
 and corresponding energy $E_{2D}(a,b) = 2\sqrt{2a} +  4\sqrt{b} $;
 \item Case $a/b\in(0,r_*]$
    \begin{equation*}
      A_{2D}= [0,\sqrt{a}]^2 , \quad B_{2D}= [0,\sqrt{a+b}]^2\setminus A_{2D},
    \end{equation*}
     with corresponding energy $E_{2D}(a,b) =  4\sqrt{a+b} + 2 \sqrt{a}$.
   \end{itemize}
\end{proposition}
An illustration of the three types of minimizers of the planar double-bubble problem  is  given in Figure \ref{types}.

\begin{figure}[h]
  \centering
  \pgfdeclareimage[width=120mm]{types}{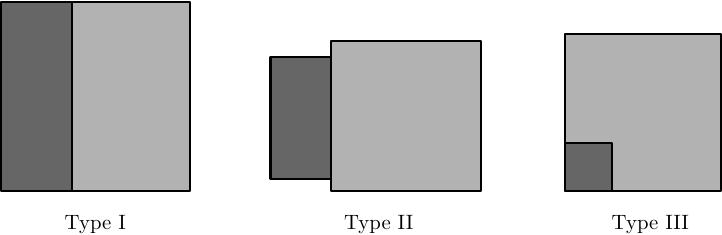}
  \pgfuseimage{types}
  \caption{Minimizers for the planar double-bubble
     problem \eqref{eq:2D_energy}.}   \label{types}
\end{figure}

\subsection{Notation and optimal lower bound}
\label{sec:notation}
 
We start by introducing some notation used throughout the
rest of the paper.

At first, let us consider the specific geometry of Figure \ref{fig}.  It
consists of the union of two cuboids sharing a square face. Indicating by $M$ the area of such a shared face, 
the cuboids have a square cross section of sidelength
$\sqrt{M}$ and have heigth $V_A/M$ and $V_B/M$, respectively. Within this
specific class of configurations, one can identify the
minimal value of the energy  by simply computing
\begin{align}\label{eq: Emin}
E_{\rm min} := \min_{M>0}   \Big(  3M   +   \frac{4(V_A +V_B)}{\sqrt{M}}   \Big). 
\end{align} 
Indeed, such configuration features seven  faces: three squares with area $M$ and four rectangular faces with sidelengths $\sqrt{M}$ and $(V_A+V_B)/M$.

Let now $(A,B)$ be a pair of disjoint sets of finite perimeter in
$\R^3$, not necessarily being cuboids,  satisfying  
 $${
r:=\frac{V_B}{V_A} \in [1/2,2].} $$ 
Choose a  plane spanned by two
coordinate directions, and denote the area of the  orthogonal
projection of $A\cup B$   onto the plane by $m$, as well as the area
of the projection of the two sets  $A$ and $B$ by $m_A$ and $m_B$, respectively. We further set 
\begin{equation}\label{eq:definitionofp}
p := \frac{m_A + m_B}{m} - 1.
\end{equation}
The value $p$ describes the size of the overlap of the
projections of $A$ and $B$ onto the chosen plane. Up to
redefining the axes, in the following we assume that such
plane is
given by $\R^2 \times \lbrace 0 \rbrace$, so that the projection
occurs in the $z$ direction. 

We consider horizontal slices $\R^2 \times \lbrace t \rbrace$ and set 
$$a(t) = \mathcal{L}^2\big(A \cap (\R^2 \times \lbrace t \rbrace)    \big), \quad b(t) = \mathcal{L}^2\big(B \cap (\R^2 \times \lbrace t \rbrace)    \big). $$
Fubini's Theorem ensures that
\begin{align}\label{volumina}
V_A = \int_\R a(t) \, {\rm d}t, \quad \quad V_B =  \int_\R b(t) \, {\rm d}t. 
\end{align}
Let $\mathcal{T}_0 = \lbrace t\colon r a(t) = b(t)>0\rbrace$, $\mathcal{T}_A = \lbrace t\colon r a(t) > b(t) \rbrace$ and $\mathcal{T}_B = \lbrace t\colon r a(t) < b(t) \rbrace$. For convenience, we define
$$U_A = \int_{\mathcal{T}_A} (a(t) + b(t)) \, {\rm d}t, \quad U_B = \int_{\mathcal{T}_B} (a(t) + b(t)) \, {\rm d}t, \quad U_0 = \int_{\mathcal{T}_0} (a(t) + b(t)) \, {\rm d}t. $$
Clearly, $V_A + V_B = U_A + U_B + U_0$.  For $t \in \mathcal{T}_A$, we
set $\alpha(t) = b(t)/a(t) \in [0, r)$ and for $t \in \mathcal{T}_B$,
we set $\beta(t) = a(t)/b(t)\in [0,1/r)$. Since $r\mathcal{L}^3(A) = rV_A = V_B =  \mathcal{L}^3(B)$, we have by the definition of $\mathcal{T}_0$
$$r \int_{\mathcal{T}_A\cup \mathcal{T}_B} a(t) \,   {\rm d}t =   \int_{\mathcal{T}_A\cup \mathcal{T}_B} b(t) \,   {\rm d}t $$
and therefore
\begin{align}\label{eq: V*} 
\int_{\mathcal{T}_A} a(t) (r-\alpha(t)) \, {\rm d}t =  \int_{\mathcal{T}_A} (r a(t) - b(t))  \,  {\rm d}t    =  \int_{\mathcal{T}_B} (b(t) - r a(t)) \, {\rm d}t  =   \int_{\mathcal{T}_B} b(t) (1-r \beta(t)) \, {\rm d}t.
\end{align}
These definitions allow us to restate the result in the planar
case from Proposition \ref{thm:mainresult_planar} as follows.


\begin{proposition}[Minimal planar energy]\label{prop:chip}
Recall the definition  of $r_*$ in formula \eqref{rstar}  and define the function $f \colon [0,\infty) \to \R$ by 
$$f(x) = \Big( 4  + 2\sqrt{\frac{x}{x+1}} \Big) \chi_{[0,r_*]} +
\frac{4 + 2\sqrt{2x}}{\sqrt{x+1}} \chi_{(r_*, 1/2]} + 2\sqrt{6}
\chi_{(1/2,1]} $$
for $x \in [0,1]$ and by $f(x) = f(1/x)$ for $x >1$. Then, for $t \in \mathcal{T}_A$ we have 
$$E_{2D}(a(t),b(t)) = \sqrt{a(t) + b(t)} f(\alpha(t)),$$
for $t\in \mathcal{T}_B$ we have
$$E_{2D}(a(t),b(t)) = \sqrt{a(t) + b(t)} f(\beta(t)), $$
and for $t\in \mathcal{T}_0$ we have
$$ E_{2D}(a(t),b(t)) = 2\sqrt{6} \sqrt{a(t) + b(t)}. $$
\end{proposition}

Along the  proof of Theorem \ref{thm:mainresult}, we make use of the explicit values of the minimal energy in the planar case in order to estimate $E$ by considering 
horizontal slices   with respect to a well-chosen coordinate direction. 
We assume that the parameter $p$ given by
\eqref{eq:definitionofp} describing the overlap between
projections of the two sets $A$ and $B$ in the $z$
direction satisfies $p \leq 1/3$. In fact, Lemma~\ref{lem:psmaller13}
shows that this is not restrictive, up to possibly relabeling the
axes. 
 The core of the proof of Theorem
\ref{thm:mainresult} consists in the following claim.

\begin{proposition}[Optimal lower bound]\label{prop:estimateforenergy}
Let $(A,B)$ be any configuration of disjoint sets with finite
perimeter, with $r\mathcal{L}^3(A)=rV_A =V_B=\mathcal{L}^3(B) $ for
$r \in  [1/2,2]$. Suppose that $p\le 1/3$. Then,
\begin{align}\label{eq:mainestimate}
 E(A,B) \geq (2+p)m +   \frac{4\sqrt{6}}{\sqrt{4+2p}\sqrt{m}} (U_A + U_B) +   \frac{2\sqrt{6}}{\sqrt{m}} U_0.
\end{align}
Moreover, the equality in \eqref{eq:mainestimate} is attained if and only if $\mathcal{L}^1(\mathcal{T}_A) = \mathcal{L}^1(\mathcal{T}_B) = 0$ and $a(t) + b(t) = m$ for  $t \in \mathcal{T}_0\setminus \mathcal{N}$, where $\mathcal{N}$ is a set of negligible $\mathcal{L}^1$-measure. 
\end{proposition}

Moving from the optimal lower bound \eqref{eq:mainestimate}, the
proof of Theorem \ref{thm:mainresult} follows by proving that the
configuration in Figure \ref{fig} is the only one (up to translations
and axis-preserving isometries) realizing the equality case. This in
particular follows by checking that actually $p=0$ for the minimizer,
so that the two sets $A$ and $B$ have disjoint projections.

Proposition \ref{prop:estimateforenergy} is proved in Subsection
\ref{sec:proof:mainresult} below. 
As a preparation, in the
next subsection we state  some   auxiliary results  
whose  proofs  are postponed to  Section
\ref{sec:proofsoflemmas}.

\subsection{Statements of auxiliary results}\label{sec:auxiliary}

First, we will state a slicing result for the double-bubble energy. To this end, we will assume that the vertical direction corresponds to the last coordinate. Letting $G \subset \mathbb{R}^3$ be any set of finite perimeter, we indicate by $G_t = G \cap (\mathbb{R}^2 \times \{t\})$ the horizontal slice at level $t$
in the $z$ direction and denote by $\pi_3 G$ the orthogonal projection of $G$ on $\R^2 \times \lbrace 0 \rbrace$.

\begin{lemma}[Slicing lemma]\label{lem:symmetrization-new}
Suppose that $A$ and $B$ are disjoint bounded sets of finite perimeter. Then,
\begin{equation*}
 E(A,B) \geq \int_{\mathbb{R}} E_{2D}(a(t), b(t)) \, {\rm d}t + 2 \mathcal{H}^{2}( \pi_3 A \cup \pi_3 B   ) + \mathcal{H}^{2}( \pi_3 A \cap \pi_3 B   ) .
\end{equation*}
\end{lemma}

Let  us state a result showing that the
assumption $p\leq 1/3$ is not restrictive, up to relabeling the
axes. Indeed, we have the following.


\begin{lemma}[Upper bound on $p$]\label{lem:psmaller13}
Suppose that $(A,B)$ is an optimal configuration. Then, we can pick a
coordinate direction such that $p \leq 1/3$, with $p$ defined in
\eqref{eq:definitionofp} on the basis of the projections $m$, $m_A$,
and $m_B$ along that
coordinate direction. 
\end{lemma}

 The last technical lemma concerns the properties of
some auxiliary functions depending on the function $f$ defined in
Proposition \ref{prop:chip}. We separated it from the proof of the
main result in order to simplify the argument, as this 
algebraic calculation simply follows from the very definition of $f$. 

\begin{lemma}[Functions $g_A^r$ and $g_B^r$]\label{lem:appendix}
Suppose that 
\begin{equation}\label{eq:assumptionsonmamb}
m_A \geq \frac{2+p}{3} m > \frac{1+2p}{3} m \geq m_B,    
\end{equation}
where $p \in [0,1/3]$. Define
\begin{align}\label{g-def}
g_A^r(\alpha) & := \frac{1+\alpha}{r-\alpha} \Big(\frac{f(\alpha)}{ \min\lbrace \sqrt{m}, \sqrt{(1+\alpha) m_A} \rbrace   } - \frac{4\sqrt{6}}{\sqrt{4+2p}\sqrt{m}} \Big), \notag  \\
g_B^r(\beta) & := \frac{1+\beta}{1-r\beta} \Big(\frac{f(\beta)}{ \min\lbrace \sqrt{m}, \sqrt{(1+\beta) m_B} \rbrace   } - \frac{4\sqrt{6}}{\sqrt{4+2p}\sqrt{m}} \Big).
\end{align}
Then, we have
\begin{align*}
 {\rm (a)}  \ \ \min_{\alpha \in [0,r]} (g_A^r(\alpha)) =  g_A^r(0)  \le 0,  \quad   \quad \quad  {\rm (b)}  \ \ \min_{\beta \in [0,\frac{1}{r}]} g_B^r(\beta) = g_B^r(0)  > 0. 
\end{align*}
\end{lemma}

\subsection{Proof of Theorem \ref{thm:mainresult}}\label{sec:proof:mainresult}

Moving from the discussion in  the beginning
of Section \ref{sec:mainproof}, the first step in the proof of
Theorem \ref{thm:mainresult} is to show that inequality
\eqref{eq:mainestimate} holds for all admissible configurations
$(A,B)$. Therefore, we first prove Proposition
\ref{prop:estimateforenergy}. Recall that the proofs of the
auxiliary Lemmas \ref{lem:symmetrization-new}--\ref{lem:appendix}
are postponed to Section \ref{sec:proofsoflemmas}. 

\begin{proof}[Proof of Proposition \ref{prop:estimateforenergy}]
{\bf Step 1.}  By Lemma \ref{lem:symmetrization-new} and the definition of the projections $m$ and $p$, a first lower bound for the energy of a configuration $(A,B)$ is given by
$$ E(A,B) \geq (2+p)m +   \int_\R E_{2D}(a(t),b(t)) \, {\rm d}t. $$
We now use the expression for  $E_{2D}(a(t),b(t))$ from 
Proposition \ref{prop:chip} to  get
\begin{align} 
E(A,B) \geq (2+p)m +   \int_{\mathcal{T}_A}  \hspace{-0.1cm}
  &\sqrt{a(t) + b(t)} f(\alpha(t))\, {\rm d}t \nonumber \\ 
&+  \int_{\mathcal{T}_B}  \hspace{-0.1cm}   \sqrt{a(t) + b(t)} f(\beta(t))\, {\rm d}t + \int_{\mathcal{T}_0} 2\sqrt{6} \sqrt{a(t) + b(t)} \, {\rm d}t. \label{eq:foursummands}
\end{align}
By definition, we have $a(t) + b(t) \leq m$ for all $t \in
\mathbb{R}$. We first estimate the last addend in \eqref{eq:foursummands} by 
\begin{equation}\label{eq:basicestimateint0}
\int_{\mathcal{T}_0} 2\sqrt{6} \sqrt{a(t) + b(t)} \, {\rm d}t \geq \int_{\mathcal{T}_0} \frac{2\sqrt{6}}{\sqrt{m}} (a(t) + b(t))\, {\rm d}t = \frac{2\sqrt{6}}{\sqrt{m}} U_0,
\end{equation}
with strict inequality if and only if  $a(t) + b(t) < m$ on a subset of $\mathcal{T}_0$ with positive $\mathcal{L}^1$-measure.

Now, we estimate the second addend in \eqref{eq:foursummands}. We again use that $a(t) + b(t) \leq m$ for all $t \in \mathbb{R}$, this time together with 
$$a(t) + b(t) \leq (1 + \alpha(t)) m_A \quad \mbox{for } t \in \mathcal{T}_A,$$
to get
\begin{align}\label{eq:basicestimateinta}
\int_{\mathcal{T}_A}   \sqrt{a(t) + b(t)} f(\alpha(t))\, {\rm d}t \geq \int_{\mathcal{T}_A}   (a(t) + b(t)) \frac{f(\alpha(t))}{ \min\lbrace \sqrt{m}, \sqrt{((1+\alpha(t)) m_A} \rbrace   }\, {\rm d}t
\end{align}
 with strict inequality if and only if 
$$a(t) + b(t) < \min\lbrace  {m},  {((1+\alpha(t)) m_A} \rbrace$$ 
on a subset of $\mathcal{T}_A$  with positive
$\mathcal{L}^1$-measure. One can verify that the function
$f(x)/\sqrt{1+x}$ is increasing on the interval $[0,2/\sqrt{3}-1]$ and
decreasing on the interval $[2/\sqrt{3}-1,1/2]$. Therefore, the
minimum is exactly attained in $0$ or $\frac12$, and one can check
that its values at $0$ and $1/2$ 
are equal, namely 
$$f(0) = 4 = \frac{f(1/2)}{\sqrt{1+1/2}}.$$ 
Moreover, we have $f(x) = 2\sqrt{6}$ for $1/2 \le x \le 2$.  Therefore,
\begin{equation}\label{strici1}
\int_{\mathcal{T}_A}   \sqrt{a(t) + b(t)} f(\alpha(t))\, {\rm d}t \geq \int_{\mathcal{T}_A}   (a(t) + b(t)) \min \left\{   \frac{2\sqrt{6}}{\sqrt{m}}, \frac{4}{\sqrt{m_A}} \right\}\, {\rm d}t,
\end{equation} 
with strict inequality if $m_A \le {2m}/{3}$ and  $a(t) + b(t) < m $ on a subset of $\mathcal{T}_A$  with  positive $\mathcal{L}^1$-measure. In fact, if  $m \le   {((1+\alpha(t)) m_A}$, we have strict inequality already in \eqref{eq:basicestimateinta}, and if  $m  >  {((1+\alpha(t)) m_A}$, we have strict inequality between the right-hand sides of \eqref{eq:basicestimateinta} and  \eqref{strici1} since  $m_A \le {2m}/{3}$ implies $\min \lbrace \frac{2\sqrt{6}}{\sqrt{m}}, \frac{4}{\sqrt{m_A}} \rbrace =    \frac{2\sqrt{6}}{\sqrt{m}}$.

We similarly estimate the third addend in \eqref{eq:foursummands}:  since $a(t) + b(t) \leq m$ for all $t \in \mathbb{R}$ and 
$$a(t) + b(t) \leq (1 + \beta(t)) m_B \quad \mbox{for } t \in \mathcal{T}_B,$$
 we have
\begin{align}\label{eq:basicestimateintb}
\int_{\mathcal{T}_B}   \sqrt{a(t) + b(t)} f(\beta(t))\, {\rm d}t \geq \int_{\mathcal{T}_B}   (a(t) + b(t)) \frac{f(\beta(t))}{ \min\lbrace \sqrt{m}, \sqrt{((1+\beta(t)) m_B} \rbrace   }\, {\rm d}t,
\end{align}
so in particular
\begin{align}\label{strici2}
\int_{\mathcal{T}_B}   \sqrt{a(t) + b(t)} f(\beta(t))\, {\rm d}t \geq     \int_{\mathcal{T}_B}   (a(t) + b(t)) \min \left\{   \frac{2\sqrt{6}}{\sqrt{m}}, \frac{4}{\sqrt{m_B}}\right\}\, {\rm d}t,  
\end{align}
with strict inequality if $m_B \le {2m}/{3}$ and  $a(t) + b(t) < m $ on a subset of $\mathcal{T}_B$  with  positive $\mathcal{L}^1$-measure. We collect these estimates and deduce from \eqref{eq:foursummands} that  
\begin{align}\label{eq:almostfinalestimatefore}
E(A,B) \ge (2+p)m + \int_{\mathcal{T}_A}   (a(t) + b(t)) &\min \left\{   \frac{2\sqrt{6}}{\sqrt{m}}, \frac{4}{\sqrt{m_A}} \right\}\, {\rm d}t \\ 
& +  \int_{\mathcal{T}_B}   (a(t) + b(t)) \min \left\{   \frac{2\sqrt{6}}{\sqrt{m}}, \frac{4}{\sqrt{m_B}} \right\}\, {\rm d}t +   \frac{2\sqrt{6}}{\sqrt{m}}U_0. \notag 
\end{align}
To simplify a later argument in 
Step 2 below, we 
rewrite the above inequality in an equivalent form. Let us
introduce the notation 
$$\tilde{p} = \sqrt{6}\sqrt{1+p/2} - \sqrt{6}$$ 
and   
$$\hat{U} = \frac{4\sqrt{6}}{\sqrt{4+2p}\sqrt{m}} (U_A + U_B) +
\frac{2\sqrt{6}}{\sqrt{m}} U_0,$$
where $\hat U$ corresponds to the last two terms in the
right-hand side of \eqref{eq:mainestimate}. By adding and
subtracting $ \frac{4\sqrt{6}}{\sqrt{4+2p}\sqrt{m}} (a(t) + b(t)) $
to  both integrals,  and recalling that  integration of $a(t) + b(t)$ over $\mathcal{T}_A$ recovers $U_A$ (and similarly for $B$), 
we can write inequality \eqref{eq:almostfinalestimatefore}  as  
\begin{align}\label{eq: 1}
E(A,B) \geq (2+p)m +   \int_{\mathcal{T}_A}  &\frac{4 (a(t) + b(t))}{\sqrt{4+2p}\sqrt{m}\sqrt{m_A}}  \min\big\{ \tilde{p}\sqrt{m_A},  \sqrt{4+2p} \sqrt{m}     - \sqrt{6m_A} \big\} \, {\rm d}t \\ &  +  \int_{\mathcal{T}_B} \frac{4 (a(t) + b(t))}{\sqrt{4+2p}\sqrt{m}\sqrt{m_B}}  \min\big\{ \tilde{p}\sqrt{m_B},  \sqrt{4+2p} \sqrt{m}     - \sqrt{6m_B} \big\} \, {\rm d}t + \hat{U}. \notag
\end{align} 



{\bf \flushleft Step 2.} In order to proceed with the proof of
\eqref{eq:mainestimate}, we shall check that the integrals on the
right-hand side of \eqref{eq: 1} are not negative. We separately consider the
following mutually exclusive alternatives: 

{\flushleft (1)}  At least one of the sets $\mathcal{T}_A$ and
$\mathcal{T}_B$ has zero measure. Then, using \eqref{eq: V*}
we get that both $\mathcal{T}_A$ and $\mathcal{T}_B$ have zero
measure, and \eqref{eq:mainestimate} follows. Moreover, the inequality is strict if and only if $a(t) + b(t) < m$ on a subset of $\mathcal{T}_0$ of positive $\mathcal{L}^1$-measure, see \eqref{eq:basicestimateint0}.

{\flushleft (2)} Both sets $\mathcal{T}_A, \mathcal{T}_B$ have positive measure and 
\begin{equation}\label{eq:extraconditioninstep2}
\sqrt{4+2p}\sqrt{m} > \sqrt{6 m_A} \quad \mbox{ and } \quad \sqrt{4+2p}\sqrt{m} > \sqrt{6 m_B}.
\end{equation}
The inequality \eqref{eq:mainestimate} follows from
\eqref{eq: 1} since the integrands are positive. Furthermore, the
inequality is strict whenever $p>0$. If  $p=0$, given that
$m_A < {2m}/{3}$ and  $m_B < {2m}/{3}$, the separate estimates for
$\mathcal{T}_0$, $\mathcal{T}_A$ and $\mathcal{T}_B$ which are given
in inequalities \eqref{eq:basicestimateint0}, \eqref{strici1}, and
\eqref{strici2} respectively, imply that the inequality is strict
whenever $a(t) + b(t) < m$ on a subset of $\mathcal{T}_0 \cup
\mathcal{T}_A \cup \mathcal{T}_B$ with positive
$\mathcal{L}^1$-measure.  Since in this case we have $m = m_A + m_B$,
equality in \eqref{eq:mainestimate} shows $a(t) = m_A$ and $b(t) =
m_B$ for a.e.\ $t$. In view of  \eqref{volumina}, this gives $r =
V_B/V_A = {m_B}/{m_A}$.  Thus, $r a(t) = b(t)$ for a.e.\ $t$, which implies $\mathcal{L}^1(\mathcal{T}_A) = \mathcal{L}^1(\mathcal{T}_B) = 0$, and we are back in case (1).

{\flushleft (3)} Both sets $\mathcal{T}_A, \mathcal{T}_B$ have positive measure and condition \eqref{eq:extraconditioninstep2} fails, i.e., we have 
$$\sqrt{4+2p}\sqrt{m} \le \sqrt{6 m_A} \quad \mbox{or} \quad \sqrt{4+2p}\sqrt{m} \le \sqrt{6 m_B},$$ which we rewrite as
\begin{align}\label{eq: max}
\max \lbrace m_A, m_B \rbrace \geq \frac{2+p}{3} m.
\end{align}
The rest of the proof (including Step 3) concerns this case. Recall the definition of the functions $g^r_A$ and  $g^r_B$ in Lemma \ref{lem:appendix}.  Using these definitions,  we can rewrite estimate \eqref{eq:basicestimateinta} as  
\begin{align*}
\int_{\mathcal{T}_A} \sqrt{a(t) + b(t)} f(\alpha(t))\, {\rm d}t \geq \int_{\mathcal{T}_A} a(t) (r-\alpha(t)) g_A^r(t) \, {\rm d}t + \frac{4\sqrt{6}}{\sqrt{4+2p}\sqrt{m}} U_A,
\end{align*}
and similarly we can rewrite \eqref{eq:basicestimateintb} as
\begin{align*}
\int_{\mathcal{T}_B} \sqrt{a(t) + b(t)} f(\beta(t))\, {\rm d}t \geq \int_{\mathcal{T}_B} b(t) (1-r\beta(t)) g_B^r(t) \, {\rm d}t + \frac{4\sqrt{6}}{\sqrt{4+2p}\sqrt{m}} U_B.
\end{align*}
 By plugging this into \eqref{eq:foursummands}, and using \eqref{eq:basicestimateint0}  for the last addend, we get
\begin{align*}
E(A,B) 
 \geq (2+p)m +   \int_{\mathcal{T}_A} a(t) (r-\alpha(t)) g_A^r(t) \, {\rm d}t \notag  +  \int_{\mathcal{T}_B}  b(t) (1-r\beta(t)) g_B^r(t) \, {\rm d}t + \hat{U}.
\end{align*}

 {\bf \flushleft Step 3.}  From now on, we assume without restriction that $m_A \geq m_B$. Hence, $m_B = (1+p)m - m_A$, and by \eqref{eq: max} we get 
$m_B \leq \frac{1+2p}{3} m$.  From  Lemma \ref{lem:appendix} we have   
\begin{align*}
{\rm (i)}  \ \ \min_{\beta \in [0, 1/r)} g^r_B(\beta) = g^r_B(0) >0, \quad   {\rm (ii)}  \ \ \max_{\alpha \in [0, r )} (-g^r_A(\alpha)) = - g^r_A(0) \ge  0,
\end{align*}
and thus
\begin{align*}
\inf_{t \in \mathcal{T}_B} g^r_B(\beta(t))  -\sup_{t\in \mathcal{T}_A} ( -g^r_A(\alpha(t))) \ge  g^r_B(0)  + g^r_A(0).
\end{align*}
As $g^r_A(0)  \le 0 $ for each $r$,  we get by monotonicity  in $r$ that 
\begin{align*} 
\inf_{t \in \mathcal{T}_B} g^r_B(\beta(t))  -\sup_{t\in \mathcal{T}_A} ( -g^r_A(\alpha(t))) \geq g^{1/2}_B(0)  + g^{1/2}_A(0) =   \frac{4}{\sqrt{m_B}} +  \frac{8}{\sqrt{m_A}} - \frac{12\sqrt{6}}{\sqrt{4+2p}\sqrt{m}}.
\end{align*}
By the definition of $p$, we have $m_B = m (1+p) - m_A$,  and thus 
\begin{align}
\inf_{t \in \mathcal{T}_B} g_B^r(\beta(t))  -\sup_{t\in \mathcal{T}_A} ( -g_A^r(\alpha(t))) &\geq \frac{1}{\sqrt{m}} \Big( \frac{4}{\sqrt{1+p-\frac{m_A}{m}}} +  \frac{8}{\sqrt{\frac{m_A}{m}}} - \frac{12\sqrt{6}}{\sqrt{4+2p}} \Big).
\end{align}
We minimize the sum of the first two addends in terms of ${m_A}/{m} \in [0,1]$ and get
\begin{align}
\inf_{t \in \mathcal{T}_B} g_B^r(\beta(t))  -\sup_{t\in \mathcal{T}_A} ( -g_A^r(\alpha(t))) \geq \frac{1}{\sqrt{m}} \Big( \frac{4(1+2^{{2}/{3}})^{{3}/{2}}}{\sqrt{1+p}  } - \frac{12\sqrt{6}}{\sqrt{4+2p}}   \Big).
\end{align}
Optimizing with respect to $p \in [0,{1}/{3}]$  we conclude 
\begin{align*}
c_2:=\inf_{t \in \mathcal{T}_B} g_B^r(\beta(t)) > \sup_{t\in \mathcal{T}_A} (- g_A^r(\alpha(t)))=:c_1.
\end{align*}
 Then, denoting the value of the integrals in \eqref{eq: V*} by $U_*$, using that $r-\alpha \ge 0  $ on $\mathcal{T}_A$ and $1-r \beta \ge 0$ on $\mathcal{T}_B$ 
we get from Step 2 that 
\begin{align*}
E(A,B) 
& \geq (2+p)m +   \int_{\mathcal{T}_A} a(t) (r-\alpha(t)) g_A^r(t) \, {\rm d}t \notag  +  \int_{\mathcal{T}_B}  b(t) (1-r\beta(t)) g_B^r(t) \, {\rm d}t + \hat{U} \\
& = (2+p)m   - \int_{\mathcal{T}_A} a(t) (r-\alpha(t))    (-g_A^r(t)) \, {\rm d}t + \int_{\mathcal{T}_B}  b(t) (1-r\beta(t)) g_B^r(t) \, {\rm d}t \notag  + \hat{U} \\
& \geq (2+p)m + c_2 \int_{\mathcal{T}_B}  b(t) (1-r\beta(t)) \, {\rm d}t - c_1 \int_{\mathcal{T}_A} a(t) (r-\alpha(t)) \, {\rm d}t \notag  + \hat{U} \\
& = (2+p)m + (c_2 - c_1) U_* + \hat{U}.
\end{align*}
Recall that the term $\hat{U}$ is exactly the one appearing
in \eqref{eq:mainestimate}. Therefore, whenever $U_* > 0$, the
inequality in \eqref{eq:mainestimate} is strict. Hence, in the case of equality, we have $U_* =
0$, but by the definition of $U_*$ this implies that
$\mathcal{L}^1(\mathcal{T}_A) = \mathcal{L}^1(\mathcal{T}_B) = 0$, and
we are back to case (1) from Step 2 above.  This concludes the proof. 
\end{proof}

Having established Proposition \ref{prop:estimateforenergy}, we
can proceed to the proof of Theorem \ref{thm:mainresult}. 

\begin{proof}[Proof   of Theorem \ref{thm:mainresult}]
As 
$U_A+ U_B + U_0 =   V_A + V_B  $ and $p \geq 0$, we use
\eqref{eq:mainestimate} in order to estimate
\begin{align*}
E(A,B) \geq (2+p)m + \frac{4\sqrt{6}}{\sqrt{4+2p}\sqrt{m}} (V_A + V_B),
\end{align*}
where an equality is possible only if $p = 0$. Then, by the change of
variables $M = (2+p)m/3$, and optimizing with respect to all possible
values $M >0$ and $p \in [0,1/3]$, we find $E(A,B) \geq E_{\rm min} $  (see \eqref{eq: Emin})  with equality only if $\mathcal{L}^1(\mathcal{T}_A) = \mathcal{L}^1(\mathcal{T}_B) = 0$ and $a(t) + b(t) = m$ for $t \in \mathcal{T}_0 \setminus \mathcal{N}$.  Optimizing with respect to $M$ in \eqref{eq: Emin} indeed gives the minimal energy given in Theorem \ref{thm:mainresult}. Moreover, we observe that this energy is attained by the configuration indicated in Theorem \ref{thm:mainresult}.

We can characterize ground states uniquely as follows: the above
argument shows that any minimizer necessarily has $U_A=U_B = 0$,
$U_0=V_A + V_B$, and $p=0$. This yields that the projection of $A$ and
$B$ must have empty intersection
and each slice with $t \in \mathcal{T}_0 \setminus \mathcal{N}$ has the
same geometry. By the planar double-bubble result of Proposition
\ref{thm:mainresult_planar}, this geometry is then given by
two specific rectangles joined at one face, namely the
configuration from Theorem \ref{thm:mainresult}.  
\end{proof}

\section{Proofs of the  auxiliary results}\label{sec:proofsoflemmas}

\subsection{Proof of Lemma \ref{lem:symmetrization-new}}\label{sec:proof:lemsymmetrization-new}

Let us recall a classical slicing result for rectifiable sets, 
see for instance \cite[Section 18.3]{Maggi}. Suppose that $F
\subset \mathbb{R}^3$ is a
rectifiable set with $\mathcal{H}^2(F) < +\infty$. Recall that $F_t$ is  the horizontal slice of the set $F$ at level $t$ in direction $x_3$, i.e.,
\begin{equation*}
F_t = F \cap \{ (x_1,x_2,t): (x_1,x_2) \in \mathbb{R}^{2} \}.
\end{equation*}
 In a similar fashion, we let
$${F}^{(x_1,x_2)} :=   F \cap \{ (x_1,x_2,y): y \in \mathbb{R}\},$$
and get that $\mathcal{H}^0(F^{(x_1,x_2)})$ is finite for almost every $(x_1,x_2)$. For every Borel function $g\colon \mathbb{R}^3 \rightarrow [-\infty,\infty]$ with $g \geq 0$ or $g \in  L^1(\mathbb{R}^3)$ we have
\begin{align}\label{eq:maggi}
\int_{F} g \, \sqrt{1 - (\nu_{F} \cdot e_3)^2} \, {\rm d}\mathcal{H}^{2} = \int_{\mathbb{R}}  \int_{F_t} g \, {\rm d}\mathcal{H}^1 \, {\rm d}t,
\end{align}
where $\nu_F$ denotes the unit normal to   $F$, as well as 
\begin{align}\label{eq:maggi-different}
\int_{F}   |\nu_{F} \cdot e_3|\, {\rm d}\mathcal{H}^{2} = \int_{\R^2}  \mathcal{H}^0\big( F^{(x_1,x_2)} \big) \, {\rm d}(x_1,x_2).
\end{align}
 We are now ready to prove the slicing Lemma
 \ref{lem:symmetrization-new}. 

\begin{proof}[Proof of Lemma \ref{lem:symmetrization-new}]
Let $G := A \cup B$ and  $F := \partial^* A \cup \partial^* B$. We split the $\ell_1$-perimeter of $F $ into a `vertical' and `horizontal' part. To be exact, we write
\begin{equation*}
E(A,B) = \int_{F} |\nu_F|_1 \, {\rm d} \mathcal{H}^{2} = \int_{F} (|\nu'_F|_1 + |(\nu_F)_3|) \, {\rm d}\mathcal{H}^{2},
\end{equation*}
where we denote $x = (x',x_3)$ with $x' \in \mathbb{R}^{2}$. Recall that we use the notation $|\cdot|_1$ to denote the $\ell_1$-norm of a vector both in two and three dimensions.  The `vertical' part and the `horizontal' part is given by the integration of $|\nu'_F|_1 = |(\nu_F)_1| + |(\nu_F)_2|$  and  $|(\nu_F)_3|$ over $F$, respectively. To this end, we introduce the function
\begin{equation*}
\bar{g} = \frac{|(\nu_F)_1| + |(\nu_F)_2|}{\sqrt{1 - (\nu_F \cdot e_3)^2}}.
\end{equation*}
We use \eqref{eq:maggi} with $g = \bar{g}$ on $F$ (and $0$ otherwise)
and get
\begin{equation*}
\int_{F} |\nu'_F|_1 \, {\rm d}\mathcal{H}^{2} = \int_{F} \bar{g} \, \sqrt{1 - (\nu_F \cdot e_3)^2} \, {\rm d}\mathcal{H}^{2} = \int_{\mathbb{R}} \int_{F_t} \bar{g} \, {\rm d}\mathcal{H}^{1}  {\rm d}t = \int_\mathbb{R}  \ell_1(F_t)  \, {\rm d}t,
\end{equation*}
where in the last step we used the fact that $\frac{1}{ \sqrt{1 - (\nu_F \cdot e_3)^2}} \nu_F' \in \R^2$ is a unit normal to $F_t$.  Since for a.e. $t \in \mathbb{R}$ we have $F_t = (\partial^* A \cup \partial^* B)_t$, the value $\ell_1(F_t)$ corresponds to the double-bubble energy of the configuration $(A_t, B_t)$, and consequently we get 
$$ \ell_1(F_t)  \ge E_{2D}(a(t),b(t)),$$
because the areas of $A_t$ and $B_t$ are $a(t)$ and $b(t)$, respectively.  Therefore,  
\begin{equation}\label{eq:estimateofg1}
\int_{F} |\nu'_F|_1 \, {\rm d} \mathcal{H}^{2} \ge \int_\mathbb{R}  E_{2D}(a(t),b(t)) \, {\rm d}t.
\end{equation}
On the other hand,  the `horizontal' term can be  estimated in terms of the area of the largest horizontal slice.  By \eqref{eq:maggi-different} we get 
\begin{equation*}
\int_{F} |(\nu_F)_3| \, {\rm d}\mathcal{H}^{2} = \int_{\R^2}  \mathcal{H}^0\big( F^{(x_1,x_2)} \big) \, {\rm d}(x_1,x_2).
\end{equation*}
For  $\mathcal{H}^2$-a.e.\ $(x_1,x_2) \in (\pi_3 A \cup \pi_3 B)
\setminus (\pi_3 A \cap \pi_3 B  )$ we have  $ \mathcal{H}^0(
(\partial^* {G})^{(x_1,x_2)} ) \ge 2$, and  for
$\mathcal{H}^2$-a.e.\ $(x_1,x_2) \in   (\pi_3 A \cap \pi_3 B  )$ we
have  $ \mathcal{H}^0( (\partial^* {G})^{(x_1,x_2)} ) \ge 3$. This shows
\begin{equation}\label{eq:estimateofg2}
\int_{F} |(\nu_F)_3| \, {\rm d}\mathcal{H}^{2} = \int_{\R^2}  \mathcal{H}^0\big( F^{(x_1,x_2)} \big) \, {\rm d}(x_1,x_2) \ge  2 \mathcal{H}^{2}( \pi_3 A \cup \pi_3 B   ) + \mathcal{H}^{2}( \pi_3 A \cap \pi_3 B   ) . 
\end{equation}
Combining  \eqref{eq:estimateofg1} and \eqref{eq:estimateofg2} concludes the proof. 
\end{proof}

\subsection{Proof of Lemma \ref{lem:psmaller13}}\label{sec:proof:lemppsmaller13}
Denote by $(m_1,m_2,m_3)$ the areas of projections of $A \cup B$ on all coordinate directions, by $(m_1^A,m_2^A,m_3^A)$ the areas of projections of $A$, and by $(m_1^B,m_2^B,m_3^B)$ the areas of projections of $B$. Then, for $i = 1,2,3$ let $p_i = ({m_i^A + m_i^B})/{m_i} - 1$ and suppose by contradiction that $p_i > 1/3$ for all $i$.

Letting $F = \partial^* A \cup \partial^* B$, by arguing as in  the proof of inequality  \eqref{eq:estimateofg2} we get
\begin{align*}
E(A,B) = \int_{F}  |\nu_F|_1  \, {\rm d}\mathcal{H}^2 = \sum_{i=1}^3  \int_{F} |(\nu_F)_i| \, {\rm d}\mathcal{H}^{2}  \ge  \sum_{i=1}^3 \big(  2 \mathcal{H}^{2}( \pi_i A \cup \pi_i B   ) + \mathcal{H}^{2}( \pi_i A \cap \pi_i B   )  \big),
\end{align*}
where $\pi_i$ denotes the orthogonal projection on the plane with normal vector $e_i$. Then, we get 
\begin{equation*}
E(A,B) \geq \sum_{i=1}^3 (2m_i + p_i m_i) > \sum_{i=1}^3 (2m_i + \frac{1}{3} m_i) = \frac{7}{3} \sum_{i=1}^3 m_i.
\end{equation*}
Now, if we let $\overline{m} = ({m_1 + m_2 + m_3})/{3}$, we have
\begin{equation*}
E(A,B) > \frac{7}{3} \sum_{i=1}^3 m_i = 7\overline{m}.
\end{equation*}
But this is the double-bubble energy of the following configuration $(\hat{A},\hat{B})$: the set $\hat{A} \cup \hat{B}$ is a cube with area of each side equal to $\overline{m}$, both sets $\hat{A}$ and $\hat{B}$ are cuboids, and the interface between them is a square of area $\overline{m}$ which is parallel to one of the sides of the original cube. The placement of the interface is such that the volume ratio is preserved. By the isoperimetric inequality for the $\ell_1$-norm, the volume of $\hat{A} \cup \hat{B}$ is greater or equal to the volume of $A \cup B$. Thus, since  $E(A,B)  >E(\hat{A},\hat{B}) $, the original configuration $(A,B)$ was not optimal: a contradiction.

\subsection{Proof of Lemma \ref{lem:appendix}}\label{sec:proof:lemappendix}
We start by observing that   $g_A^r(0)\le 0$ and  $g_B^{r}(0) > 0$. In fact, using that ${m_A}/{m} \ge  ({2+p})/{3}$ (see assumption \eqref{eq:assumptionsonmamb}), we   get 
$$ \frac{4\sqrt{6}}{\sqrt{4+2p}} - \frac{4}{\sqrt{{m_A}/{m}}} \ge  \frac{4\sqrt{6}}{\sqrt{4+2p}} - \frac{4\sqrt{3}}{\sqrt{2+p}} =0  $$
all $p \in [0,{1}/{3}]$ which shows that  $g_A^r(0)\le 0$ since $f(0) = 4$. In a similar fashion, $g_B^{r}(0) > 0$ follows from 
$$  \frac{4}{\sqrt{{m_B}/{m}}} - \frac{4\sqrt{6}}{\sqrt{4+2p}}  \ge   \frac{4\sqrt{3}}{\sqrt{1+2p}} - \frac{4\sqrt{6}}{\sqrt{4+2p}} >0 $$
for all $p \in [0,{1}/{3}]$, where we used ${m_B}/{m} \le  {(1+2p)}/{3}$, see \eqref{eq:assumptionsonmamb}.

The main part of the proof consists now in checking that $g_A^r$ and
$g_B^r$ attain their minima at $0$. The proof is structured as
follows. In Step 1 we first show that the problem can be
reduced to the cases $r =1/2$ and $r=1$. In Step 2 we
introduce several auxiliary functions and use their specific
properties to prove the statement. The proof of these properties is then given in Steps 3--8.

{\bf \flushleft Step 1.}  Let us reduce the problem to specific values of $r$: we claim that it suffices to show
\begin{itemize}
\item[(i)] $\min_{\alpha \in [0, 1]} g_A^1(\alpha) = g_A^1(0)  $, 
\item[(ii)] $\min_{\alpha \in [0,  {1}/{2} ]} g_A^{ 1/2}(\alpha) =  g_A^{ 1/2}(0)$,
\item[(iii)] $\min_{\beta \in [0, 2]} g_B^{1/2}(\beta) =  g_B^{1/2}(0)$. 
\end{itemize}
 We assume for the moment that  the conditions (i)--(iii) hold,
 and show the statement of Lemma~\ref{lem:appendix}.  For simplicity, we use the abbreviation  
$$v(x,y) = \frac{f(x)}{ \min\lbrace \sqrt{m}, \sqrt{(1+x) y} \rbrace
} - \frac{4\sqrt{6}}{\sqrt{4+2p}\sqrt{m}}.$$
We start by proving (a) of Lemma \ref{lem:appendix}.  As $g^r_A(0) \le 0$, it is not restrictive to consider only $\alpha \in [0,r]$ with $g^r_A(\alpha) \le 0$.   Suppose first that $r \ge 1$. We write  
\begin{align*}
g^r_A(\alpha) = \frac{1}{r}   \frac{1+\alpha}{1 -\frac{\alpha}{r}} v(\alpha,m_A) = \frac{1}{r}   \frac{1-\alpha}{1 -\frac{\alpha}{r}}        \frac{1+\alpha}{1 -\alpha} v(\alpha,m_A) =  \frac{1}{r}   \frac{1-\alpha}{1 -\frac{\alpha}{r}} g_A^1(\alpha).
\end{align*}
If  $\alpha \in [0,1]$,  we have   $0 \le  \frac{1-\alpha}{1 -\frac{\alpha}{r}}  \le 1$. This, along with the above relation  and (i), implies  that 
$$0  \ge g_A^1(\alpha) \ge g_A^1(0),$$ and consequently 
$$g^r_A(\alpha) \ge   \frac{1}{r} g^1_A(0) = g^r_A(0).$$
If instead $1 \le \alpha \le r \le 2$, by Proposition \ref{prop:chip} and  assumption  \eqref{eq:assumptionsonmamb} we have
$$g^r_A(\alpha) = \frac{1+\alpha}{r-\alpha} \Big( \frac{2\sqrt{6}}{ \sqrt{m}  } - \frac{4\sqrt{6}}{\sqrt{4+2p}\sqrt{m}} \Big) \ge  0 . $$
Thus, the minimum of $g_A^r$ is attained at $\alpha = 0$ with $g_A^r(0) \le 0$.

 On the other hand, if ${1}/{2} \le r <1$, we first write
$${{g^r_A(\alpha) = \frac{1}{r}   \frac{1+\alpha}{1 -\frac{\alpha}{r}} v(\alpha,m_A)
            =  \frac{1}{r}   \frac{\frac{1}{2}-\alpha}{1 -\frac{\alpha}{r}}       \frac{1+\alpha}{\frac{1}{2}-\alpha}    v(\alpha,m_A)  =   \frac{1}{r}   \frac{\frac{1}{2}-\alpha}{1 -\frac{\alpha}{r}}      g_A^{ 1/2}(\alpha).} }$$
  If $\alpha \in [0,{1}/{2}]$,   we have  $0 \le \frac{1/2-\alpha}{1 -\alpha/r}  \le {1}/{2}$. This, together  with the above relation  and (ii), shows 
$${ 0 \ge g_A^{1/2}(\alpha) \ge g_A^{1/2}(0),}$$ 
and then also   
$${g^r_A(\alpha) \ge  \frac{1}{2r}  g^{1/2}_A(0) = g^r_A(0).}$$
If instead ${1}/{2} \le \alpha \le r \le 1$, by Proposition \ref{prop:chip} and  assumption  \eqref{eq:assumptionsonmamb} we   have
$${g^r_A(\alpha) = \frac{1+\alpha}{r-\alpha} \Big( \frac{2\sqrt{6}}{ \sqrt{m}  } - \frac{4\sqrt{6}}{\sqrt{4+2p}\sqrt{m}} \Big)  \ge 0, }$$
and thus the minimum of $g_A^r$ is attained at $\alpha = 0$ with $g_A^r(0) \le 0$.

Let us now come to the proof of (b)  of Lemma \ref{lem:appendix}. We first write
\begin{align*}
g^r_B(\beta) =   \frac{1+\beta}{1 -r\beta} v(\beta,m_B) =    \frac{1-\frac{\beta}{2}}{1 -r\beta}  \frac{1+\beta}{1-\frac{\beta}{2}}    v(\beta,m_B) =       \frac{1-\frac{\beta}{2}}{1 -r\beta}   g_B^{1/2}(\beta).
\end{align*}
Recall that $g_B^{1/2}(0) > 0$. Then, for $\beta \in [0,{1}/{r}]$, by $\frac{1-\beta/2}{1 -r\beta} \ge 1$ and  $g_B^{1/2}(\beta) \ge g_B^{1/2}(0)  > 0$ (see (iii)), we conclude
$$g^r_B(\beta) \ge g_B^{1/2}(0) = g_B^r(0),$$
i.e., $g_B^r$ attains its minimum at $\beta = 0$.

{\bf \flushleft Step 2.}  We now proceed with the proof of the properties (i)--(iii). To this end, we  define the auxiliary functions
$$h_1^r(x) = \frac{1+x}{r-x} \Big( \frac{\sqrt{6}}{\sqrt{1+\frac{p}{2}}} - \frac{f(x)}{2} \Big),$$
$$h_2^r(x) = \frac{1+x}{1-rx} \Big( \frac{\sqrt{6}}{\sqrt{1+\frac{p}{2}}} - \frac{f(x)}{2} \Big),$$
$$h_3^r(x) = \frac{1+x}{r-x} \Big( \frac{\sqrt{6}}{\sqrt{1+\frac{p}{2}}} - \frac{\sqrt{m}}{\sqrt{m_A}} \frac{f(x)}{2 \sqrt{1+ x} } \Big),$$
and
$$h_4^r(x) = \frac{1+x}{1-rx} \Big( \frac{\sqrt{6}}{\sqrt{1+\frac{p}{2}}} - \frac{\sqrt{m}}{\sqrt{m_B}} \frac{f(x)}{2 \sqrt{1+ x} } \Big).$$
 Recalling the definition  of the function $g_A^r$  in \eqref{g-def}, we have
\begin{equation*}
-g_A^r(\alpha) = \frac{2}{\sqrt{m}} \min(h_1^r(\alpha), h_3^r(\alpha)).
\end{equation*}
In other words, it is given by $h_3^r(\alpha)$ for $\alpha \in
[0,{m}/{m_A} - 1]$ and by $h_1^r(\alpha)$ for $\alpha \in [{m}/{m_A} -
1, r]$. In  Steps 3--5  below we show for $r =1/2$ or $r=1$ that   $h_1^r$ is decreasing and $h_3^r$ achieves its maximum on the interval $[0,{m}/{m_A} - 1]$ at $0$. This will show that $-g_A^r$ is maximized  (and thus $g_A^r$ is minimized)  at $0$ for $r =1/2$ or $r=1$. In a similar fashion, we have
\begin{equation*}
g_B^{1/2}(\beta) = \frac{2}{\sqrt{m}} \max(-h_2^{1/2}(\beta),-h_4^{1/2}(\beta)).
\end{equation*}
In other words, it is given by $-h_4^{1/2}(\beta)$ for $\beta \in [0,{m}/{m_B} - 1]$ and by $-h_2^{1/2}(\beta)$ for $\beta \in [{m}/{m_B} - 1, {1}/{r}]$. In Steps 6--8 below we show that  $h_2^{1/2}$ is decreasing,  so that  $-h_2^{1/2}$ is increasing,  and that $h_4^{1/2}$ attains its maximum at $0$. This is enough to conclude that   $g_B^{1/2}$ is minimized at $0$.

{\bf \flushleft Step 3.}   We first check that   $h_1^r$ is decreasing
on $[0,r]$. For later purposes, we derive this property not only for
$r =1/2$ and $r=1$ but also for $r=2$.  Recall the
value $r_*$ defined in \eqref{rstar}.  We will check separately
the cases $x \in [0,r_* ]$, $x \in [r_* ,1/2]$, and $x \in [1/2,r]$,    where the case $x \in [1/2,r]$ is only necessary for $r =1$ or $r=2$.  For $x \in [0, r_* ]$, we have $f(x) = 4 + 2 \sqrt{{x}/({1+x})}$, so
$$h_1^r(x) = \frac{1+x}{r-x} \Big( \frac{\sqrt{6}}{\sqrt{1+\frac{p}{2}}} - 2 -  \frac{\sqrt{x}}{\sqrt{1+x}} \Big). $$
Let $C_1 = \frac{\sqrt{6}}{\sqrt{1+{p}/{2}}} - 2 \leq
\sqrt{6} - 2$. For  $r =1/2, 1$, or $2$,   we   compute the derivative and get
\begin{align*}
(h_1^r)'(x) &   = \frac{  2C_1(1+r) \sqrt{x(x+1)} - (1+2r)x - r}{2(x-r)^2 \sqrt{x(x+1)}} .
\end{align*}
 A direct calculation yields 
$$2C_1(1+r) \sqrt{x(x+1)} - (1+2r)x - r < 0$$ 
for all $x \in [0,r_* ]$ and $r =1/2, 1$, or $2$,  so that $h_1^r$ is decreasing on the interval $[0, r_*]$.  
%


%
%
%
%
%

Similarly, for $x \in [r_*,1/2]$, we have $f(x) = (4 + 2\sqrt{2x})/\sqrt{1+ x}$, so
$$h_1^r(x) = \frac{1+x}{r-x} \Big( \frac{\sqrt{6}}{\sqrt{1+\frac{p}{2}}} - \frac{2}{\sqrt{1+x}} -  \frac{\sqrt{2x}}{\sqrt{1+x}} \Big). $$
Denote by $C_2 = \frac{\sqrt{6}}{\sqrt{1+p/2}} \leq \sqrt{6}$. For $r \in \lbrace 1/2,1,2\rbrace$, we directly calculate the derivative and get
\begin{align*}
(h_1^r)'(x) &= \frac{(2+2r) C_2 \sqrt{x (x + 1)} - 2 \sqrt{x} (x + 2+r) - \sqrt{2} ((1+2r) x + r)}{2 (x - r)^2 \sqrt{x(x + 1)}} \\
&\leq \frac{(2+2r) \sqrt{6} \sqrt{x (x + 1)} - 2 \sqrt{x} (x + 2+r) - \sqrt{2} ((1+2r) x + r)}{2 (x - r)^2 \sqrt{x(x + 1)}}  \le 0
\end{align*} 
with equality if and only if $x = 1/2$. Hence,  $h_1^r$ is also decreasing on the interval $[r_*,1/2]$, and consequently we have shown that $h_1^r$ is decreasing on the whole interval $[0,1/2]$.

%
%
%
%
%
%

%
%
%

Finally, for $x \in [1/2,r]$, we just need to consider the cases $r=1$ and $r=2$. We have $f(x) = 2 \sqrt{6}$, so
$$ h_1^1(x) = \frac{1+x}{r-x} \Big( \frac{\sqrt{6}}{\sqrt{1+\frac{p}{2}}} - \sqrt{6} \Big). $$
 We compute the derivative and get
\begin{align*}
(h_1^1)'(x) =  \Big( \frac{\sqrt{6}}{\sqrt{1+\frac{p}{2}}} - \sqrt{6}\Big) \frac{1+r}{(r-x)^2} \le 0.
\end{align*}
Hence, $h_1^r$ is decreasing on $[0,r]$ for $r =1/2, 1$, or $2$.  

{\bf \flushleft Step 4.} Now, we focus on $h_3^r$. By
\eqref{eq:assumptionsonmamb}, we have ${m}/{m_A} - 1 \leq 1/2$. In
this step we show that $h_3^r$ is maximized on $[0,{m}/{m_A} - 1 ]$ at
one of the three points $0$,  ${m}/{m_A} - 1$, or $r_* $ (and that
the latter is only possible if $r_* \le {m}/{m_A} - 1 $).   The values at the three points will then be compared in Step 5.  To this end, we will analyze the monotonicity of $h_3^r$. More precisely, we check that in the intervals $ [0, r_* ]$ and $[r_* ,1/2]$ the function $(h_3^r)'$ changes sign at most once and, if it does, it changes from minus to plus  (as $x$ increases). This indeed shows that the maximum is attained at $0$,  ${m}/{m_A} - 1$, or $r_* $.

Let us come to the details. For $x \in [0, r_* ]$, we have $f(x) = 4 + 2 \sqrt{{x}/({1+x})}$, so
$$ h_3^r(x) = \frac{1+x}{r-x} \Big( \frac{\sqrt{6}}{\sqrt{1+\frac{p}{2}}} - \frac{\sqrt{m}}{\sqrt{m_A}} \cdot \frac{2}{\sqrt{1+ x}} - \frac{\sqrt{m}}{\sqrt{m_A}} \cdot \frac{\sqrt{x}}{1+ x} \Big).$$
For  $r =1/2$ or $r=1$, we directly compute the first derivative and get
\begin{align*}
(h_3^r)'(x) &= \frac{\sqrt{6}}{\sqrt{1+\frac{p}{2}}} \Big( \frac{1+x}{r-x} \Big)' - \frac{\sqrt{m}}{\sqrt{m_A}} \Big( \frac{1+x}{r-x} \cdot \frac{2}{\sqrt{1+ x} } \Big)' - \frac{\sqrt{m}}{\sqrt{m_A}} \Big( \frac{1+x}{r-x} \cdot \frac{\sqrt{x}}{1 + x} \Big)' \\
&= \frac{\sqrt{6}}{\sqrt{1+\frac{p}{2}}} \cdot \frac{1+r}{(r-x)^2} - \frac{\sqrt{m}}{\sqrt{m_A}} \cdot \frac{r+x+2}{(r-x)^2 (1+x)^{1/2}} - \frac{\sqrt{m}}{\sqrt{m_A}} \cdot \frac{r+x}{2\sqrt{x} (r-x)^2  }\\
&= (r-x)^{-2} \Big(  \frac{\sqrt{6}}{\sqrt{1+\frac{p}{2}}} (1+r) - \frac{\sqrt{m}}{\sqrt{m_A}} \cdot \frac{r+x+2}{  (1+x)^{1/2}} - \frac{\sqrt{m}}{\sqrt{m_A}} \cdot \frac{r+x}{2\sqrt{x}    } \Big).
\end{align*}
Note that
$${x \mapsto -\frac{r+x+2}{  (1+x)^{1/2}} -
\frac{r+x}{2\sqrt{x}    }}$$
is increasing on $[0, r_* ]$,
and thus $(h_3^r)'$ can change sign at most once (from negative
 to positive 
 as $x$ increases).  For $x \in [r_* , 1/2]$, we have $f(x) = ({4 + 2\sqrt{2x}})/{\sqrt{1+ x}}$, and thus
$$ h_3^r(x) = \frac{1+x}{r-x} \Big( \frac{\sqrt{6}}{\sqrt{1+\frac{p}{2}}} - \frac{\sqrt{m}}{\sqrt{m_A}} \cdot \frac{2}{1+ x} - \frac{\sqrt{m}}{\sqrt{m_A}} \cdot \frac{\sqrt{2x}}{1+ x} \Big).$$
For $r =1/2$ or $r=1$, we calculate  the first derivative and obtain
\begin{align*}
(h_3^r)'(x) &= \frac{\sqrt{6}}{\sqrt{1+\frac{p}{2}}} \Big( \frac{1+x}{r-x} \Big)' - \frac{\sqrt{m}}{\sqrt{m_A}} \Big( \frac{1+x}{r-x} \cdot \frac{2}{1+ x} \Big)' - \frac{\sqrt{m}}{\sqrt{m_A}} \Big( \frac{1+x}{r-x} \cdot \frac{\sqrt{2x}}{1 + x} \Big)' \\
&= \frac{\sqrt{6}}{\sqrt{1+\frac{p}{2}}} \cdot \frac{1+r}{(r-x)^2} - \frac{\sqrt{m}}{\sqrt{m_A}} \cdot \frac{2}{(r-x)^2} - \frac{\sqrt{m}}{\sqrt{m_A}} \cdot \frac{r+x}{\sqrt{2} (r-x)^2 x^{1/2}}\\
&= (r-x)^{-2} \Big( \frac{\sqrt{6}}{\sqrt{1+\frac{p}{2}}} (1+r) - 2\frac{\sqrt{m}}{\sqrt{m_A}}  -  \frac{\sqrt{m}}{\sqrt{m_A}}  \frac{r+x}{ \sqrt{2x} }\Big).
\end{align*}
We note that $x \mapsto - 2   -   ({r+x})/{ \sqrt{2x}}$ is increasing on $[r_* ,1/2]$, i.e.,\ $(h_3^r)'$ can change sign at most once.

{\bf \flushleft Step 5.} As seen in Step 4, $h_3^r$ attains its maximum on the interval $[0,{m}/{m_A} - 1]$ at one of the points $0$, $r_* $, and ${m}/{m_A} - 1$, where $r_* $ is only possible if $r_* \le {m}/{m_A} - 1$. Using $f(r_* ) = \frac{20}{41}(7+2\sqrt{2})$, we compute explicitly the three values and get
\begin{equation*}
h_3^r(0) = \frac{\sqrt{6}}{r\sqrt{1+\frac{p}{2}}} - 2 \frac{\sqrt{m}}{r\sqrt{m_A}},
\end{equation*}
\begin{equation*}
h_3^r(r_* ) =  \frac{1 + r_* }{r-r_* } \Big(\frac{\sqrt{6}}{\sqrt{1+\frac{p}{2}}} -  \frac{\sqrt{m}}{\sqrt{m_A}} \frac{10(7+2\sqrt{2})}{41\sqrt{1+r_* }} \Big),
\end{equation*}
and
\begin{align*}
h_3^r \bigg( \frac{m}{m_A} - 1 \bigg) & =
                                        \frac{\frac{m}{m_A}}{1+r-\frac{m}{m_A}}
                                        \Big(
                                        \frac{\sqrt{6}}{\sqrt{1+\frac{p}{2}}}
                                        - \frac{\sqrt{m}}{\sqrt{m_A}}
                                        \frac{f(\frac{m}{m_A} - 1)}{2
                                        \sqrt{\frac{m}{m_A}} } \Big)
  \\
  &= \frac{\frac{m}{m_A}}{1+r-\frac{m}{m_A}} \Big( \frac{\sqrt{6}}{\sqrt{1+\frac{p}{2}}} - \frac{f(\frac{m}{m_A} - 1)}{2} \Big).
\end{align*}
To see that $h_3^r(0) \ge h_3^r(r_* )$ in the case $r_*
\le {m}/{m_A}-1$, it suffices to observe that, for $r =1/2$ or
$r=1$,  
\begin{align*}
\Big( \frac{10(7+2\sqrt{2})}{41} \frac{\sqrt{1+r_* }}{r-r_* } - \frac{2}{r} \Big) \frac{\sqrt{m}}{\sqrt{m_A}} & \ge \Big( \frac{10(7+2\sqrt{2})}{41}  \frac{\sqrt{1+r_* }}{r-r_* } - \frac{2}{r} \Big)\sqrt{r_* +1} \\
&\ge  \sqrt{6} \Big(\frac{1 + r_* }{r-r_* }  - \frac{1}{r} \Big)
 \ge  \Big(\frac{1 + r_* }{r-r_* }  - \frac{1}{r} \Big) \frac{\sqrt{6}}{\sqrt{1+\frac{p}{2}}},  
\end{align*}
%
where in the first inequality we used that ${m}/{m_A} \ge r_* + 1 $, and the second inequality can be checked by an elementary computation. In a second step, we now check that the number
\begin{equation*}
h_3^r(0) - h_3^r \bigg( \frac{m}{m_A} - 1 \bigg) = \frac{\sqrt{6}}{\sqrt{1+\frac{p}{2}}} \bigg( \frac{1}{r} - \frac{\frac{m}{m_A}}{1+r-\frac{m}{m_A}} \bigg) - 2 \frac{\sqrt{m}}{r\sqrt{m_A}} + \frac{\frac{m}{m_A}}{1+r-\frac{m}{m_A}} \frac{f(\frac{m}{m_A} - 1)}{2}
\end{equation*}
is nonnegative. Recall that  by assumption \eqref{eq:assumptionsonmamb} we have that ${m}/{m_A}  \in [1,3/2]$. We distinguish two cases depending on whether ${m}/{m_A} - 1 $ lies in $[0,r_* ]$ or $[r_* , 1/2]$. First, for ${m}/{m_A} - 1 \in [0,r_* ]$, using the explicit formula for the minimal two-dimensional energy given in Proposition \ref{prop:chip},  we have
\begin{align*}
h_3^r(0) - h_3^r \bigg( \frac{m}{m_A} - 1 \bigg) &= \frac{\sqrt{6}}{\sqrt{1+\frac{p}{2}}} \bigg( \frac{1}{r}- \frac{\frac{m}{m_A}}{1+r-\frac{m}{m_A}} \bigg) - 2 \frac{\sqrt{m}}{r\sqrt{m_A}} + \frac{\frac{m}{m_A}}{1+r-\frac{m}{m_A}} \bigg(2 + \frac{\sqrt{\frac{m}{m_A} - 1}}{\sqrt{\frac{m}{m_A}}} \bigg).
\end{align*}
We can look at the above expression as a function of a single
parameter ${m}/{m_A}$, i.e.,  define
\begin{equation*}
h_3^{*}(x) = \frac{\sqrt{6}}{\sqrt{1+\frac{p}{2}}} \bigg( \frac{1}{r} - \frac{x}{1+r-x} \bigg) - \frac{2\sqrt{x}}{r} + \frac{x}{1+r-x} \bigg(2 + \frac{\sqrt{x - 1}}{\sqrt{x}} \bigg),
\end{equation*}
so that $h_3^{*}(m/m_A) = h_3^r(0) - h_3^r  ( m/m_A - 1 )$. Since $\frac{1}{r} \le \frac{x}{1+r-x}$ for $x \ge 1$, we get
\begin{align*}
h_3^{*}(x) \geq   {\sqrt{6}}  \bigg( \frac{1}{r} - \frac{x}{1+r-x} \bigg) - \frac{2\sqrt{x}}{r} + \frac{x}{1+r-x} \bigg(2 + \frac{\sqrt{x - 1}}{\sqrt{x}} \bigg).
\end{align*}
This function is positive on $(1,3/2]$ and equal to zero at
$1$ in both cases $r =1/2$ and $r=1$. Hence, 
$$h_3^r(0) \geq h_3^r  \bigg( \frac{m}{m_A} - 1 \bigg).$$ 
%
In the second case, i.e., for $m/m_A - 1 \in [r_* ,1/2]$, again using the explicit formula given in Proposition~\ref{prop:chip} we have that $h_3^r(0) - h_3^r  ( m/m_A - 1  ) $ can be written as
\begin{align*}
 \frac{\sqrt{6}}{\sqrt{1+\frac{p}{2}}} \bigg( \frac{1}{r} - \frac{\frac{m}{m_A}}{1+r-\frac{m}{m_A}} \bigg) - 2 \frac{\sqrt{m}}{r\sqrt{m_A}} + \frac{\frac{m}{m_A}}{1+r-\frac{m}{m_A}} \bigg(\frac{2}{\sqrt{\frac{m}{m_A}}} + \frac{\sqrt{2(\frac{m}{m_A} - 1)}}{\sqrt{\frac{m}{m_A}}} \bigg).
\end{align*}
Again, we can look at the above expression as a function of a single parameter $m/m_A$, i.e., we define
\begin{equation*}
h_3^{**}(x) = \frac{\sqrt{6}}{\sqrt{1+\frac{p}{2}}} \bigg( \frac{1}{r}- \frac{x}{1+r-x} \bigg) - \frac{2\sqrt{x}}{r} + \frac{x}{1+r-x} \bigg(\frac{2}{\sqrt{x}} + \frac{\sqrt{2(x - 1)}}{\sqrt{x}} \bigg),
\end{equation*}
so that $h_3^{**}(m/m_A) = h_3^r(0) - h_3^r ( m/m_A - 1 )$. Since $\frac{1}{r} \le \frac{x}{1+r-x}$ for $x \ge 1$, we find 
\begin{align*}
h_3^{**}(x) \geq \sqrt{6} \bigg( \frac{1}{r}- \frac{x}{1+r-x} \bigg) - \frac{2\sqrt{x}}{r}+ \frac{x}{1+r-x} \bigg(\frac{2}{\sqrt{x}} + \frac{\sqrt{2(x - 1)}}{\sqrt{x}} \bigg).
\end{align*}
%
%
For $r =1/2$ or $r=1$, this function is positive on $(1,3/2)$ and equal to zero at $1$ and $3/2$. Hence, we   have $h_3^r(0) \geq h_3^r ( m/m_A - 1)$, so $h_3^r$ attains its maximum on the interval $[0,m/m_A - 1]$ at $0$.

{\bf \flushleft Step 6.} To see that $h_2^{1/2}$ is decreasing, it suffices to note that 
$$h_2^{1/2}(x) = \frac{1+x}{1-\frac{x}{2}} \Big( \frac{\sqrt{6}}{\sqrt{1+\frac{p}{2}}} - \frac{f(x)}{2} \Big) = 2\frac{1+x}{2-x} \Big( \frac{\sqrt{6}}{\sqrt{1+\frac{p}{2}}} - \frac{f(x)}{2} \Big) = 2h_1^{2}(x),$$
and to use that $h_1^2$ is decreasing, see Step 3.

{\bf \flushleft Step 7.} In this step we show that  $h_4^{1/2}$
attains its maximum on $[0, {m}/{m_B} - 1 ]$  at  one of the points
$0$, $r_* $, $1/2$, or $2$.  The values at the three
points will then be compared in Step 8.  Similarly to the argument
used in Step 4, here  the argument relies on the fact that in the
intervals $ [0,r_* ]$, $[r_* ,1/2]$, and $[1/2,2]$ the
function $(h_4^{1/2})'$ changes sign at most once (from negative 
to positive  as $x$ increases). This  indeed shows that the maximum is attained at $0$,  $r_* $, $1/2$, or $2$. As in Step 4, we check separately the cases $x \in [0,r_* ]$, $x \in [r_* ,1/2]$, and $x \in [1/2,2]$. 

For $x \in [0,r_* ]$, we have $f(x) = 4 + 2 \sqrt{ {x}/({1+x})}$, so
$$ h_4^{1/2}(x) = \frac{1+x}{1-\frac{x}{2}} \Big( \frac{\sqrt{6}}{\sqrt{1+\frac{p}{2}}} - \frac{\sqrt{m}}{\sqrt{m_B}} \cdot \frac{2}{\sqrt{1+ x}} - \frac{\sqrt{m}}{\sqrt{m_B}} \cdot \frac{\sqrt{x}}{1+ x} \Big).$$
We directly compute the derivative and get
\begin{align*}
(h_4^{1/2})'(x) &= \frac{\sqrt{6}}{\sqrt{1+\frac{p}{2}}} \Big( \frac{1+x}{1-\frac{x}{2}} \Big)' - \frac{\sqrt{m}}{\sqrt{m_B}} \Big( \frac{1+x}{1-\frac{x}{2}}  \frac{2}{\sqrt{1+ x} } \Big)' - \frac{\sqrt{m}}{\sqrt{m_B}} \Big( \frac{1+x}{1-\frac{x}{2}}  \frac{\sqrt{x}}{1 + x} \Big)' \\
&= (2-x)^{-2} \Big( \frac{6\sqrt{6}}{\sqrt{1+\frac{p}{2}}}  - \frac{\sqrt{m}}{\sqrt{m_B}}  \frac{2x+8}{ \sqrt{1+x}} - \frac{\sqrt{m}}{\sqrt{m_B}}  \frac{x+2}{  \sqrt{x}}\Big).
\end{align*}
We note that $${x \mapsto -  \frac{2x+8}{ \sqrt{1+x}} -  \frac{x+2}{  \sqrt{x}}}$$
 is increasing on $[0,r_* ]$ and thus $(h_4^{1/2})'$ can change sign at most once. For $x \in [r_* ,1/2]$, we have $f(x) = ({4 + 2\sqrt{2x}})/{\sqrt{1+ x}}$, and thus
$$ h_4^{1/2}(x) = \frac{1+x}{1-\frac{x}{2}} \Big( \frac{\sqrt{6}}{\sqrt{1+\frac{p}{2}}} - \frac{\sqrt{m}}{\sqrt{m_B}}  \frac{2}{1+ x} - \frac{\sqrt{m}}{\sqrt{m_B}}  \frac{\sqrt{2x}}{1+ x} \Big).$$
We directly compute the derivative and get
\begin{align*}
(h_4^{1/2})'(x) &= \frac{\sqrt{6}}{\sqrt{1+\frac{p}{2}}} \Big( \frac{1+x}{1-\frac{x}{2}} \Big)' - \frac{\sqrt{m}}{\sqrt{m_B}} \Big( \frac{1+x}{1-\frac{x}{2}}  \frac{2}{1+ x} \Big)' - \frac{\sqrt{m}}{\sqrt{m_B}} \Big( \frac{1+x}{1-\frac{x}{2}}  \frac{\sqrt{2x}}{1 + x} \Big)' \\
&= (2-x)^{-2} \Big( \frac{6\sqrt{6}}{\sqrt{1+\frac{p}{2}}}  - 4\frac{\sqrt{m}}{\sqrt{m_B}}  - \frac{\sqrt{m}}{\sqrt{m_B}}  \frac{\sqrt{2}(x+2)}{  \sqrt{x}} \Big).
\end{align*}
We observe that $$ x \mapsto - 4   -     \frac{\sqrt{2}(x+2)}{  \sqrt{x}}$$ is increasing on $[r_* , 1/2]$, i.e.,\ $(h_4^{1/2})'$ can change sign at most once.  

Eventually, for $x \in [1/2,2]$, we have
$$ h_4^{1/2}(x) = \frac{1+x}{1-\frac{x}{2}} \Big( \frac{\sqrt{6}}{\sqrt{1+\frac{p}{2}}} - \frac{\sqrt{m}}{\sqrt{m_B}}  \frac{\sqrt{6}}{\sqrt{1+ x}} \Big)$$
and the derivative reads as
$$ (h_4^{1/2})'(x) =    (2-x)^{-2} \Big( \frac{6\sqrt{6}}{\sqrt{1+\frac{p}{2}}}  -  \frac{\sqrt{m}}{\sqrt{m_B}}  \frac{x+4}{  \sqrt{1+x}} \Big). $$
We observe that $x \mapsto -({x+4})/{  \sqrt{1+x}}$ is increasing  on $[1/2,2]$.

{\bf \flushleft Step 8.} In this step, we compare the 
values of $h_4^{1/2}$  at  $0$, $r_*$, $1/2$, and $2$ in order to conclude that
the maximum in $[0,m/m_B- 1 ]$ is indeed at $0$.

We recall that $f(r_* ) = \frac{20}{41}(7+2\sqrt{2})$ and we compute explicitly  
\begin{equation*}
h_4^{1/2}(0) = \frac{\sqrt{6}}{\sqrt{1+\frac{p}{2}}} - 2 \frac{\sqrt{m}}{\sqrt{m_B}}, 
\end{equation*}
\begin{equation*} 
 h_4^{1/2}(r_* ) =  \frac{1 + r_* }{1-r_* /2} \Big(\frac{\sqrt{6}}{\sqrt{1+\frac{p}{2}}} -  \frac{\sqrt{m}}{\sqrt{m_B}} \frac{10(7+2\sqrt{2})}{41\sqrt{1+r_* }} \Big),
\end{equation*}
and
\begin{equation*}
h_4^{1/2} \bigg(\frac{1}{2} \bigg) =  2 \Big(\frac{\sqrt{6}}{\sqrt{1+\frac{p}{2}}} -  2\frac{\sqrt{m}}{\sqrt{m_B}} \Big).
\end{equation*}
As $0 \le p\le 1/3$  and $m_B/m \le \frac{1+2p}{3}$ (see \eqref{eq:assumptionsonmamb}),   we note that 
\begin{align*}
h_4^{1/2}(0) \le \frac{\sqrt{6}}{\sqrt{1+\frac{p}{2}}} - 2 \frac{\sqrt{3}}{\sqrt{1+2p}} < 0. 
\end{align*} 
This directly shows  $h_4^{1/2}(0) \ge h_4^{1/2}(\frac{1}{2})$. To see that $h_4^{1/2}(0) \ge h_4^{1/2}(r_* )$,  use again the assumption $p \le 1/3$ and \eqref{eq:assumptionsonmamb}  to see  $m/m_B\ge \frac{3}{1+2p} \ge \frac{9}{5}.$
Therefore,  it is elementary to check 
\begin{align*}
\Big( \frac{10(7+2\sqrt{2})}{41} \frac{\sqrt{1+r_* }}{1-r_* /2} - 2 \Big) \frac{\sqrt{m}}{\sqrt{m_B}} & \ge \Big(\frac{10(7+2\sqrt{2})}{41}  \frac{\sqrt{1+r_* }}{1-r_* /2} -2\Big) \sqrt{\frac{9}{5}}\ge  \sqrt{6} \Big(\frac{1 + r_* }{1-r_* /2}  - 1 \Big) \\
&  \ge  \Big(\frac{1 + r_* }{1-r_* /2}  -1 \Big) \frac{\sqrt{6}}{\sqrt{1+\frac{p}{2}}}.  
\end{align*}
Eventually, we show that $h_4^{1/2}(0) \ge \limsup_{x \to 2}h_4^{1/2}(x)$,  which is a bit more delicate  since $h_4^{1/2}$ is not defined at $x=2$. Recalling that for   $x \in [1/2,2]$ we have $f(x) = 2\sqrt{6}$, it holds 
\begin{align}\label{lasst} h_4^{1/2}(x) = \frac{1+x}{1-\frac{x}{2}} \Big( \frac{\sqrt{6}}{\sqrt{1+\frac{p}{2}}} - \frac{\sqrt{m}}{\sqrt{m_B}}  \frac{\sqrt{6}}{\sqrt{1+ x}} \Big).
\end{align}
Since $x \le m/m_B- 1$ and thus $m/m_B \ge 1+x$, we get $h_4^{1/2}(x) \le 0$ with strict inequality for $p>0$. In particular, for $p>0$, we get $\lim_{x\to 2} h_4^{1/2}(x) = -\infty$. Now, suppose that $p=0$, and recall by   \eqref{eq:assumptionsonmamb} that $m_B /m \le 1/3$. If $m_B /m  < 1/3$, the term on the right-hand side of \eqref{lasst} is again negative for $x$ close to 2 leading to $\lim_{x\to 2} h_4^{1/2}(x) = -\infty$. If $p=0$ and $m_B /m  = 1/3$ we calculate 
$$\lim_{x \to 2} h_4^{1/2}(x) =  \lim_{x \to 2}     \frac{1+x}{1-\frac{x}{2}} \Big(  \sqrt{6}  -  \frac{\sqrt{ 3} \sqrt{6}}{\sqrt{1+ x}} \Big) =-\sqrt{6}.$$
In this case, we also have $h_4^{1/2}(0) = \sqrt{6}- 2\sqrt{3}$. This shows  $h_4^{1/2}(0) \ge \limsup_{x \to 2}h_4^{1/2}(x)$ and concludes the proof.

\section*{Acknowledgements}
MF acknowledges support of the DFG project FR 4083/3-1. This work was
supported by the Deutsche Forschungsgemeinschaft (DFG, German Research
Foundation) under Germany's Excellence Strategy EXC 2044-390685587,
Mathematics M\"unster: Dynamics--Geometry--Structure.  WG acknowledges
support of the FWF grants 10.55776/ESP88 and
10.55776/I4354. U.S. acknowledges support of the FWF grants
10.55776/F65, 10.55776/I4354, 10.55776/I5149, and 10.55776/P32788.

\end{document}